\documentclass[a4paper,leqno,12pt]{article}

\usepackage[pdftex]{graphics,color}
\usepackage{graphics}
\usepackage{amsmath,bm,stmaryrd,mathtools, xy, subfigure}
\usepackage{array}
\usepackage{amscd}
\usepackage{amssymb}
\usepackage{enumerate}
\usepackage{indentfirst}
\usepackage{latexsym}
\usepackage{multicol}
\usepackage{color}
\usepackage{theorem}

\makeatletter\@addtoreset{equation}{section}\makeatother

\newtheorem{theorem}{Theorem}

\begingroup
\newtheorem{lemma}{Lemma}[section]
\newtheorem{proposition}[lemma]{Proposition}

\endgroup

\newtheorem{definition}{Definition}
{\theorembodyfont{\rmfamily}\newtheorem{remark}{Remark}[section]}

\newenvironment{proof}{\textit{Proof. }}{\hfill$\Box$}

\newcommand{\ep}{\varepsilon}

\newcommand{\ds}{\displaystyle}
\newcommand{\beq}[1]{\begin{equation} \label{#1}\ds}
\newcommand{\eeq}{\end{equation}}
\newcommand{\bml}[1]{\beq{#1} \begin{array}{c}\ds}
\newcommand{\eml}{\end{array}\eeq}
\newcommand{\beqq}{\begin{equation*}\ds}
\newcommand{\eeqq}{\end{equation*}}
\newcommand{\bmll}{\beqq \begin{array}{c}\ds}
\newcommand{\emll}{\end{array}\eeqq}
\renewcommand{\div}{{\rm div}\,}

\newcommand{\abs}[1]{\ensuremath{\left| #1 \right|}}

\def \de{\partial}

\def \ep{\varepsilon}

\newcommand{\R}{\mathbb{R}}

\newcommand{\id}{{\rm Id}}
\newcommand{\Sym}{\mathcal{S}}

\begin{document}

\author{Elisabetta Chiodaroli and Ond\v{r}ej Kreml\thanks{O.K. acknowledges the support of the GA\v CR (Czech Science Foundation) project GJ17-01694Y in the general framework of RVO: 67985840.}}
\title{Non--uniqueness of admissible weak solutions to the Riemann problem for the isentropic Euler equations}
\date{}

\maketitle

\centerline{EPFL Lausanne}

\centerline{Station 8, CH-1015 Lausanne, Switzerland}

\bigskip

\centerline{Institute of Mathematics, Czech Academy of Sciences}

\centerline{\v Zitn\'a 25, Prague 1, 115 67, Czech Republic}

 
\begin{abstract}
We study the Riemann problem for the multidimensional compressible isentropic Euler equations. 
Using the framework developed in \cite{ChDLKr} and based on the techniques of De Lellis and Sz\'ekelyhidi \cite{dls2},
we extend the results of \cite{ChKr} and prove that whenever the initial Riemann data give rise to a self-similar solution consisting of 
one admissible shock and one rarefaction wave and are not too far from lying on a simple shock wave, the problem admits also infinitely many admissible weak solutions.
\end{abstract}


\section{Introduction}
In this note we consider the Euler system of isentropic gas dynamics in two space dimensions
\begin{equation}\label{eq:Euler system}
\left\{\begin{array}{l}
\partial_t \rho + {\rm div}_x (\rho v) \;=\; 0\\
\partial_t (\rho v) + {\rm div}_x \left(\rho v\otimes v \right) + \nabla_x [ p(\rho)]\;=\; 0\\
\rho (\cdot,0)\;=\; \rho^0\\
v (\cdot, 0)\;=\; v^0 \, ,
\end{array}\right.
\end{equation}
where the unknowns $(\rho,v)$ denote the density and the velocity of the gas respectively. 
The pressure $p$ is a given function of $\rho$ satisfying the hyperbolicity condition $p'>0$.
We will work with pressure laws $p(\rho)=  \rho^\gamma$ with constant $\gamma\geq 1$. 
We also denote the space variable as $x=(x_1, x_2)\in \R^2$.

Being a hyperbolic system of conservation laws, the system \eqref{eq:Euler system} admits a single (mathematical) entropy, namely the physical total energy. 
Denoting $\ep(\rho)$ the internal energy related to the pressure through $p(r)=r^2 \varepsilon'(r)$ the \textit{entropy (energy) inequality} reads as
\begin{equation} \label{eq:energy inequality}
\de_t \left(\rho \varepsilon(\rho)+\rho
\frac{\abs{v}^2}{2}\right)+\div_x
\left[\left(\rho\varepsilon(\rho)+\rho
\frac{\abs{v}^2}{2}+p(\rho)\right) v \right]
\;\leq\; 0.
\end{equation}
We consider bounded weak solutions of \eqref{eq:Euler system} which satisfy \eqref{eq:Euler system} in the usual distributional sense.
Moreover, we say that a weak solution to \eqref{eq:Euler system} is \textit{admissible}, when it satisfies \eqref{eq:energy inequality} 
in the sense of distributions; we also call such solutions \textit{entropy solutions}. More precisely, admissible/entropy solutions are required to satisfy a slightly stronger condition,
i.e., a form of \eqref{eq:energy inequality} which involves also the initial data (see Definition 3 in \cite{ChKr}).
We refer to the monographs \cite{br} and \cite{da} for detailed treatises of the related background literature.

In the last years, jointly with Camillo De Lellis, we could prove a surprising series of results concerning non--uniqueness of admissible solutions 
to the isentropic Euler equations in more than one space dimension (see \cite{ch}, \cite{ChDLKr}, \cite{ChKr} and also \cite{ChFeKr}), thereby showing that the most popular concept
of admissible solution, the entropy inequality, fails even under quite strong assumption on the initial data.
Non--uniqueness originates from the construction of \textit{non--standard} rapidly oscillating solutions to \eqref{eq:Euler system} which are also admissible and
are built via 
subsequent versions of the method of convex integration originally developed by
De Lellis and Sz\'ekelyhidi \cite{dls2} for the incompressible Euler equations (see also \cite{sz}).
The results we present here arise as a continuation of the work done in \cite{ChDLKr} and \cite{ChKr}.

We are concerned with the Riemann problem for the system \eqref{eq:Euler system}, more specifically we consider initial data of the following particular form
\begin{equation}\label{eq:R_data}
(\rho^0 (x), v^0 (x)) := \left\{
\begin{array}{ll}
(\rho_-, v_-) \quad & \mbox{if $x_2<0$}\\ \\
(\rho_+, v_+) & \mbox{if $x_2>0$,} 
\end{array}\right. 
\end{equation}
where $\rho_\pm, v_\pm$ are constants.
The Riemann problem \eqref{eq:Euler system}-\eqref{eq:R_data} has been the building block in the construction of non--unique entropy solutions for 
the isentropic Euler equations starting from Lipschitz initial data in \cite{ChDLKr}) and also 
in \cite{ChKr} for the investigation on the effectiveness of the entropy dissipation rate criterion, as proposed by Dafermos in \cite{Da1}, for the same system of equations (see also
\cite{fe} for complemetary results on the Dafermos criterion).

It is well known that the Riemann problem \eqref{eq:Euler system}-\eqref{eq:R_data} admits self--similar solutions $(\rho,v)(x, t):= (r,w)(x_2/t)$ and that uniqueness holds
in the class of admissible solutions if we require them to be self-similar and to have locally bounded variation.
On the other hand, both in \cite{ChDLKr} and in \cite{ChKr} it is illustrated that, once these hypotheses are removed, uniqueness of admissible solutions can fail.
In particular, in \cite{ChKr} it was proven that any Riemann data whose associated self--similar solution consists of two shocks
admit also infinitely many non--standard solutions, which are admissible too and are genuinely two--dimensional (depend non--trivially on $x_1$).

In this note we aim at better understanding the relation between the structure of the Riemann data \eqref{eq:R_data} and the formation of admissible non--standard solutions 
originating from such data.
As detailed in Section $2$ of \cite{ChKr}, if we search for self--similar solutions $(\rho,v)(x, t):= (r,w)(x_2/t)$ of the Riemann problem \eqref{eq:Euler system}-\eqref{eq:R_data},
then, depending on the values of the constants $\rho_\pm, v_\pm$, we encounter different cases.
In particular, if we choose $v_ {-1}=v_{+1}$, then the first component of the self--similar velocity will remain constant for all positive times and
the relation between the left state $(\rho_-, v_{-2})$ and the right state $(\rho_+, v_{+2})$ determines the form of the self--similar solution.
If the right state lies on a \textit{simple wave} going through the left state (see Fig, \ref{fig:waves}), then the self--similar solution consists of either a single \textit{shock}
or a single \textit{rarefaction wave} as explained in Lemma $2.3$ in \cite{ChKr}.  We refer to \cite{da} for the precise definitions of shock and rarefaction waves. In Fig. \ref{fig:waves} we denote by $S_{1,3}$ and by $R_{1,3}$ the $1,3$- shock and $1,3$-rarefaction
waves through the point $(\rho_-, v_{-2})$.
\begin{figure}[h]
\begin{center}
\scalebox{0.5}{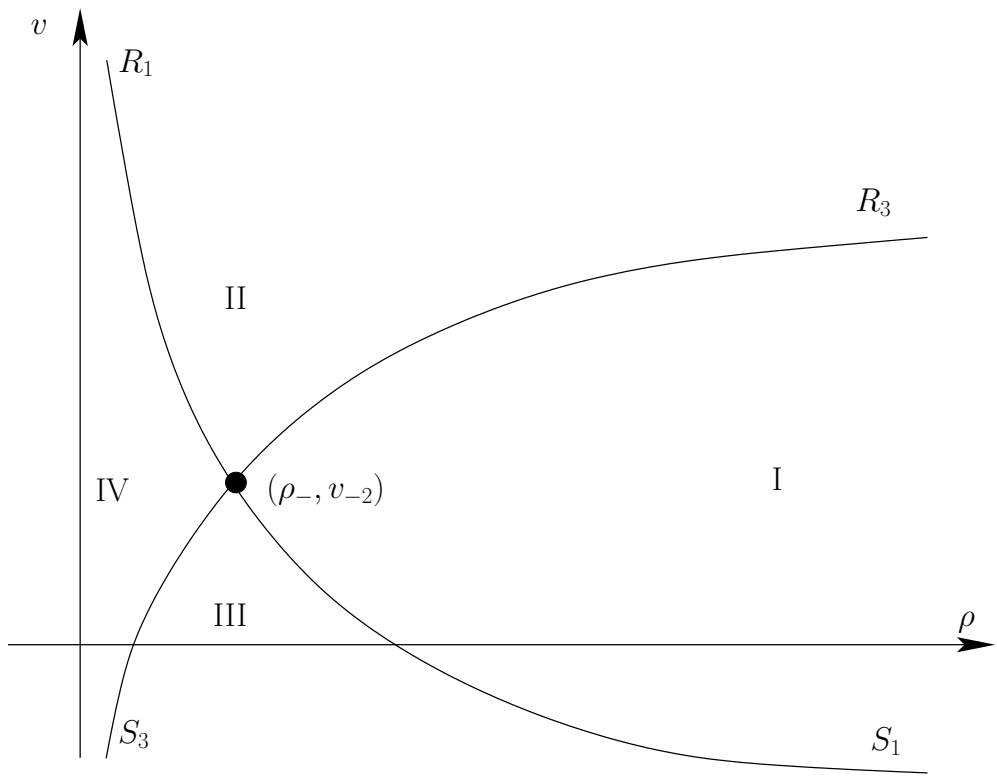}
\caption{Shocks and rarefaction curves through the point $(\rho_-, v_-)$.}
\label{fig:waves}
\end{center}
\end{figure}

If $(\rho_-, v_{-2})$ and $(\rho_+, v_{+2})$ do not lie on any simple wave, then we can distinguish four situations:
\begin{itemize}
 \item[CASE 1:]  $(\rho_+, v_{+2}) \in $ ``region I'': the solution consists of a $1$-shock and a $3$-rarefaction;
 \item[CASE 2:]  $(\rho_+, v_{+2}) \in $ ``region II'': the solution consists of two rarefaction waves;
 \item[CASE 3:]  $(\rho_+, v_{+2}) \in $ ``region III'': the solution consists of two shocks;
 \item[CASE 4:]  $(\rho_+, v_{+2}) \in $ ``region IV'': the solution consists of a $1$-rarefaction wave and a $3$-shock
\end{itemize}
In Fig. \ref{fig:case1}--\ref{fig:case4}, we describe schematically how these four cases look like placing side by side the wave curves plots 
and the pattern of the self--similar solution in the $x_2-t$ plane.
When the self--similar solution contains no discontinuities, i.e. when it consists of rarefaction waves only (CASE 2), the Riemann problem \eqref{eq:Euler system}-\eqref{eq:R_data} enjoys uniqueness as was shown first by Chen and Chen \cite{chen}. The same result was obtained in \cite{fekr}, the authors not being aware of the result of Chen and Chen. Similar results are contained also in the work of Serre \cite{serre} and related are also the work of DiPerna \cite{DP} and the works of Chen and Frid \cite{CF1} and \cite{CF2}. 
Oppositely, when the self--similar solutions consists of two shocks (CASE 3), the Riemann problem \eqref{eq:Euler system}-\eqref{eq:R_data} admits also infinitely many 
non--standard solutions as proven in \cite{ChKr}.
 \begin{figure}[hp]
 \centering
 \subfigure[Wave curves]
  {\scalebox{0.45}{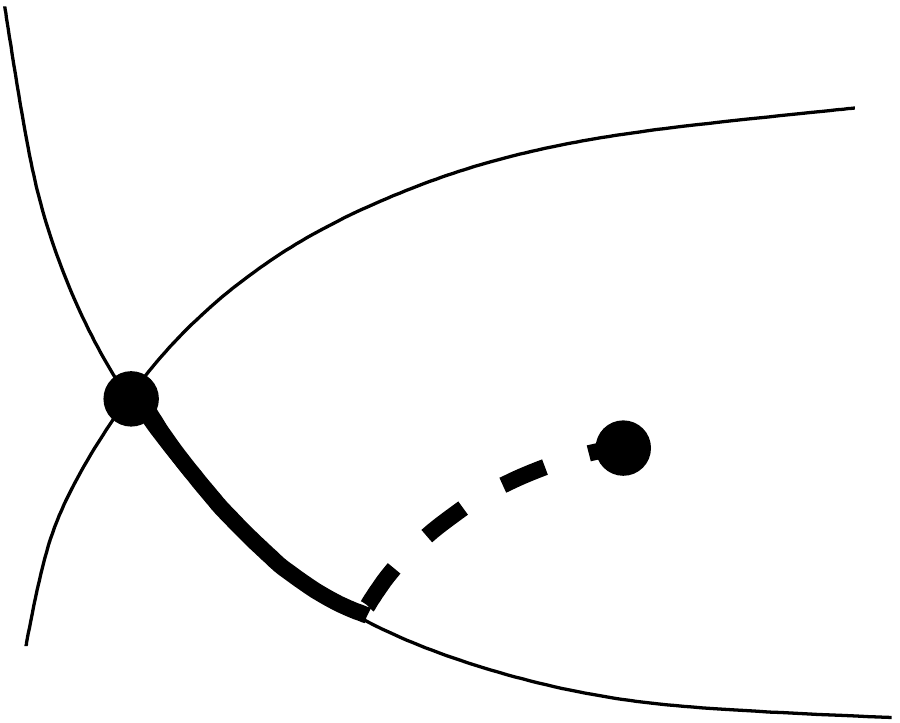}}
 \hspace{8mm}
 \subfigure[Wave fan]
 {\scalebox{0.45}{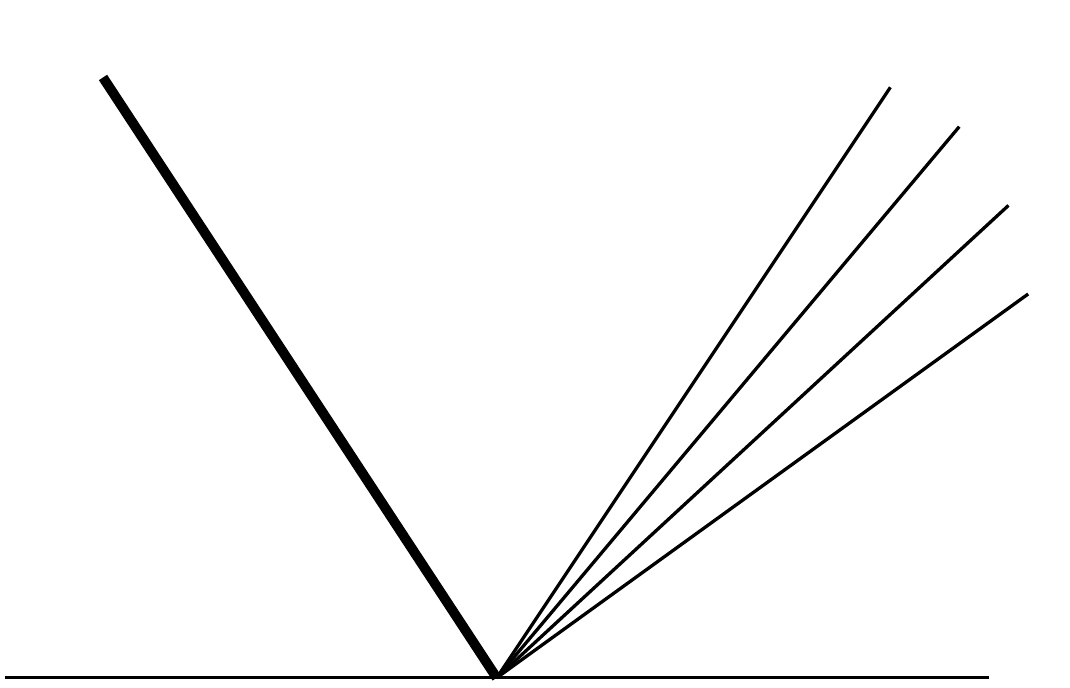}}
 \caption{Case 1.}
\label{fig:case1}
 \end{figure}
 \begin{figure}[hp]
 \centering
 \subfigure[Wave curves]
  {\scalebox{0.45}{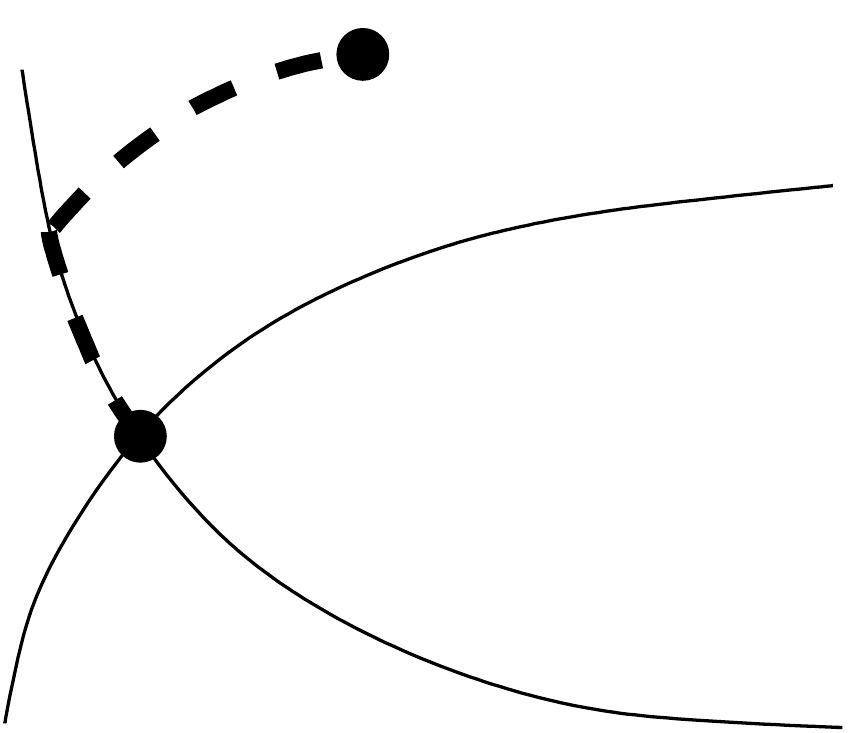}}
 \hspace{8mm}
 \subfigure[Wave fan]
 {\scalebox{0.45}{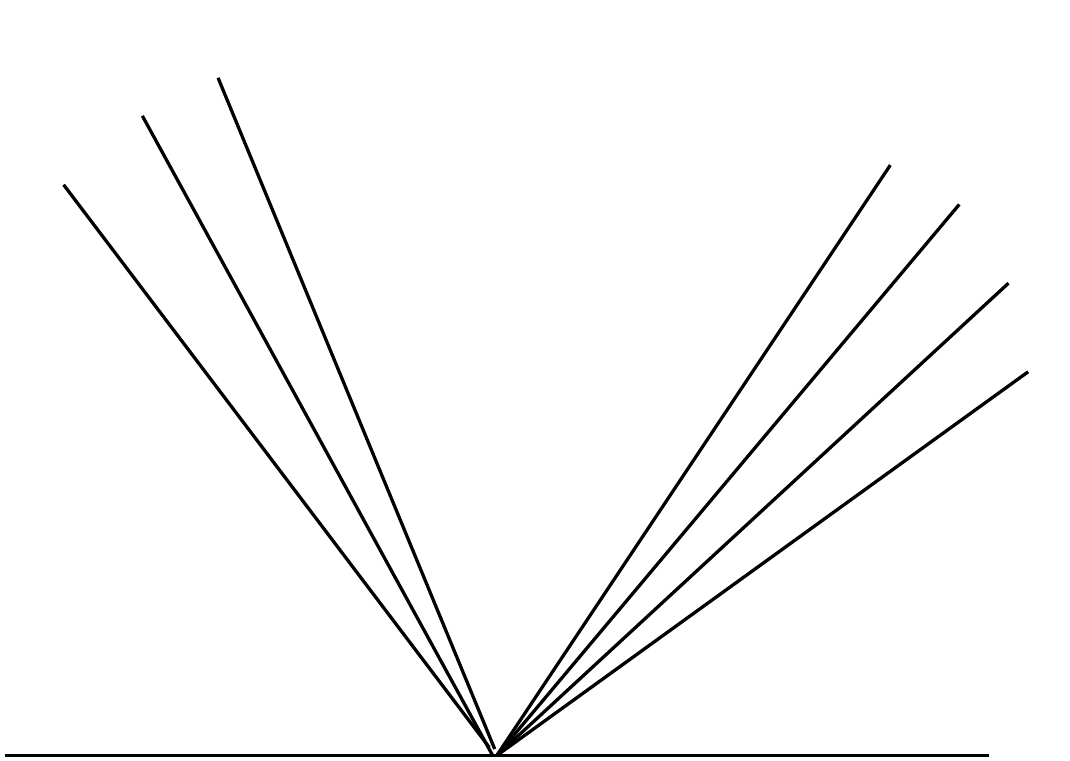}}
 \caption{Case 2.}
\label{fig:case2}
 \end{figure}
 \begin{figure}[hp]
 \centering
 \subfigure[Wave curves]
  {\scalebox{0.45}{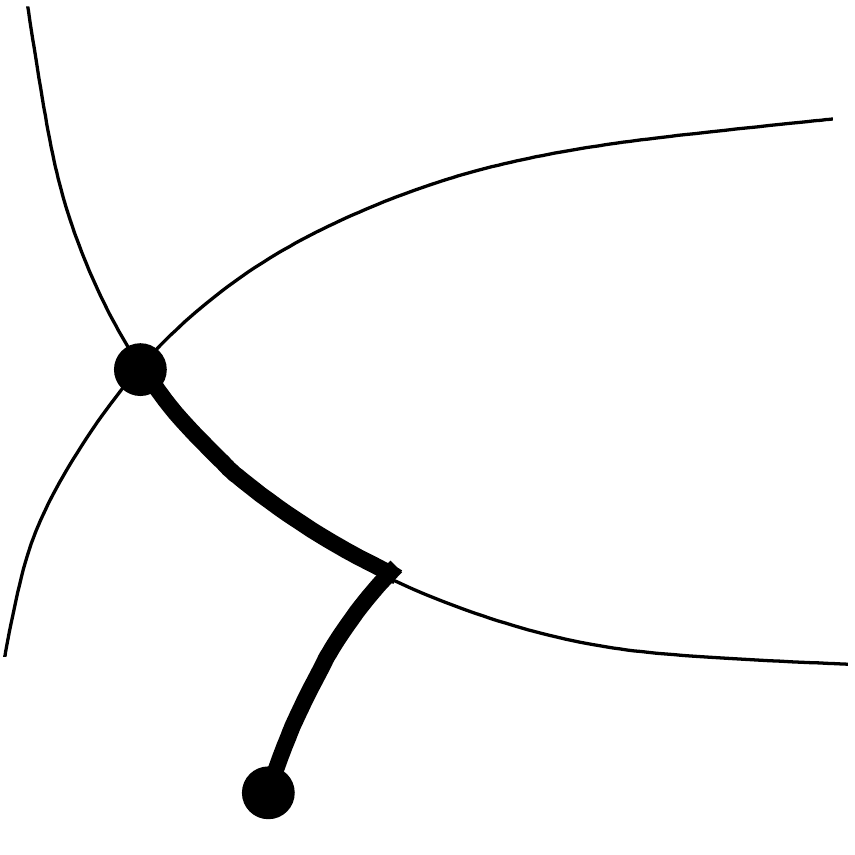}}
 \hspace{8mm}
 \subfigure[Wave fan]
 {\scalebox{0.45}{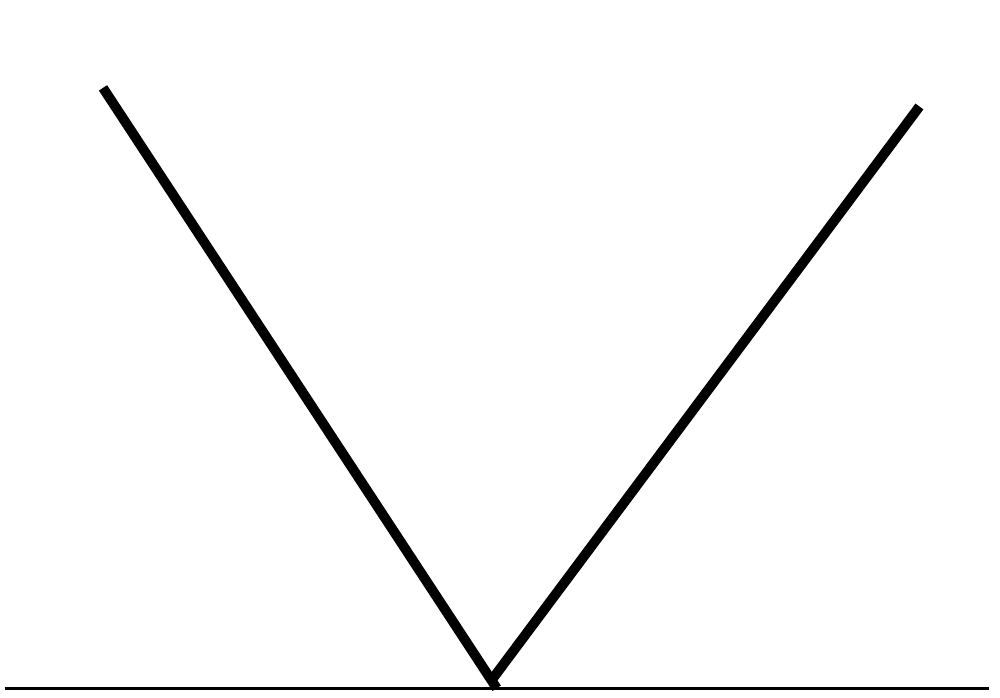}}
 \caption{Case 3.}
\label{fig:case3}
 \end{figure}
 \begin{figure}[hp]
 \centering
 \subfigure[Wave curves]
  {\scalebox{0.45}{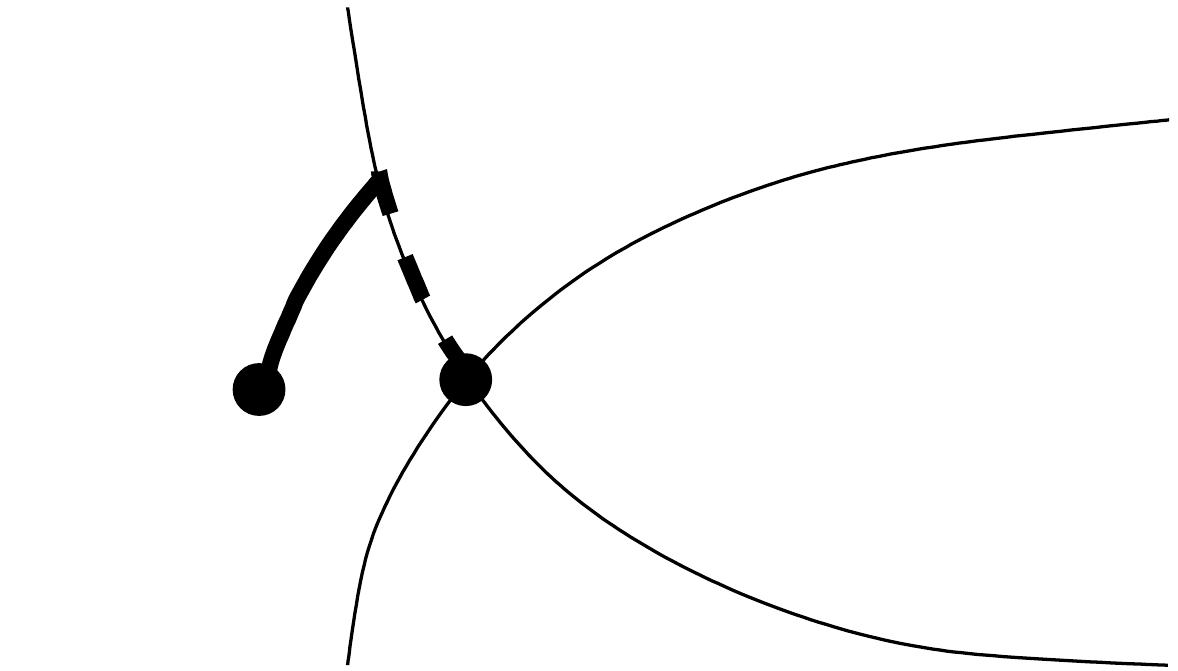}}
 \hspace{8mm}
 \subfigure[Wave fan]
 {\scalebox{0.45}{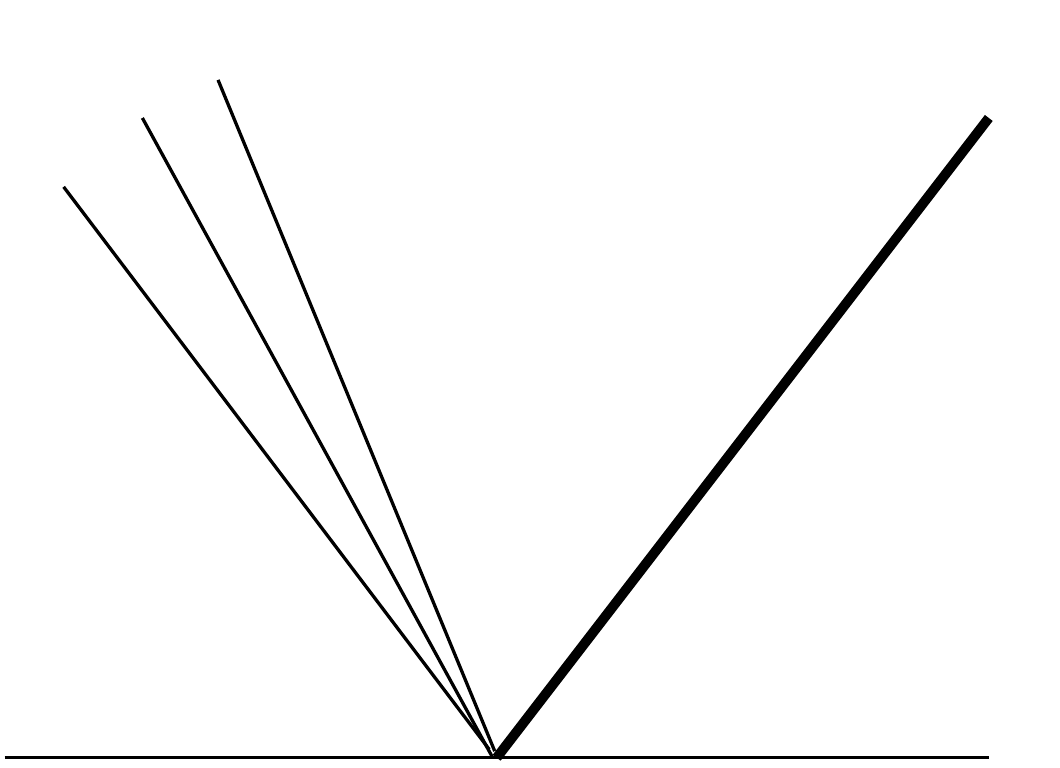}}
 \caption{Case 4.}
\label{fig:case4}
 \end{figure}

In this note we investigate whether such non--uniqueness of admissible solutions can be obtained in CASE 1 and in CASE 4 as well, at least when we are close enough 
(in a suitable sense) to CASE 3 and far from CASE 2. We do not discuss the case of Riemann data lying on a simple shock wave even if we expect that non--uniqueness should hold in this case.

Our main result proves that non--uniqueness of admissible solutions of the Riemann problem \eqref{eq:Euler system}-\eqref{eq:R_data} indeed occurs at least for right data
$(\rho_+, v_{+2})$ belonging to subregions of region I and IV which are adjacent to region III and detached from region II. The result is independent of the specific choice
for the constant $v_1$.
The precise statement of the result is as follows. 

\begin{theorem}\label{t:main}
Let $p(\rho) = \rho^\gamma$, $\gamma > 1$. Let $\rho_- \neq \rho_+$, $\rho_{\pm} > 0$ and $v_{+2} \in \mathbb{R}$ be given. 
There exists $V = V(\rho_-,\rho_+,v_{+2},\gamma) < \sqrt{\frac{(\rho_+ - \rho_-)(p(\rho_+) - p(\rho_-))}{\rho_+\rho_-}}$ 
such that for all $v_{-2}$ satisfying $V < v_{-2} - v_{+2} < \sqrt{\frac{(\rho_+ - \rho_-)(p(\rho_+) - p(\rho_-))}{\rho_+\rho_-}}$ there exists infinitely many bounded admissible weak solutions to the Euler equations \eqref{eq:Euler system} with Riemann initial data \eqref{eq:R_data}.
\end{theorem}

\begin{remark}\label{r:1}
The theorem is stated for the two-dimensional case but it naturally extends to any dimension $d > 1$.
\end{remark}

We remark that the upper bound $v_{-2} - v_{+2} < \sqrt{\frac{(\rho_+ - \rho_-)(p(\rho_+) - p(\rho_-))}{\rho_+\rho_-}}$ characterizes regions I and IV, or else said characterizes
Riemann data allowing for self--similar solutions consisting of a shock and a rarefaction wave (cf. \cite{ChKr}).
The existence of $V = V(\rho_-,\rho_+,v_{+2},\gamma)< \sqrt{\frac{(\rho_+ - \rho_-)(p(\rho_+) - p(\rho_-))}{\rho_+\rho_-}}$ and the corresponding lower bound for
$v_{-2} - v_{+2}$ guarantee instead the existence of subregions inside I and IV where non--uniqueness can arise.
In Fig. \ref{fig:region} we give a qualitative picture of such subregions in blue, while we describe in red the area where non--uniqueness holds due to \cite{ChKr}.

%

\begin{figure}[h]
\begin{center}
\scalebox{0.5}{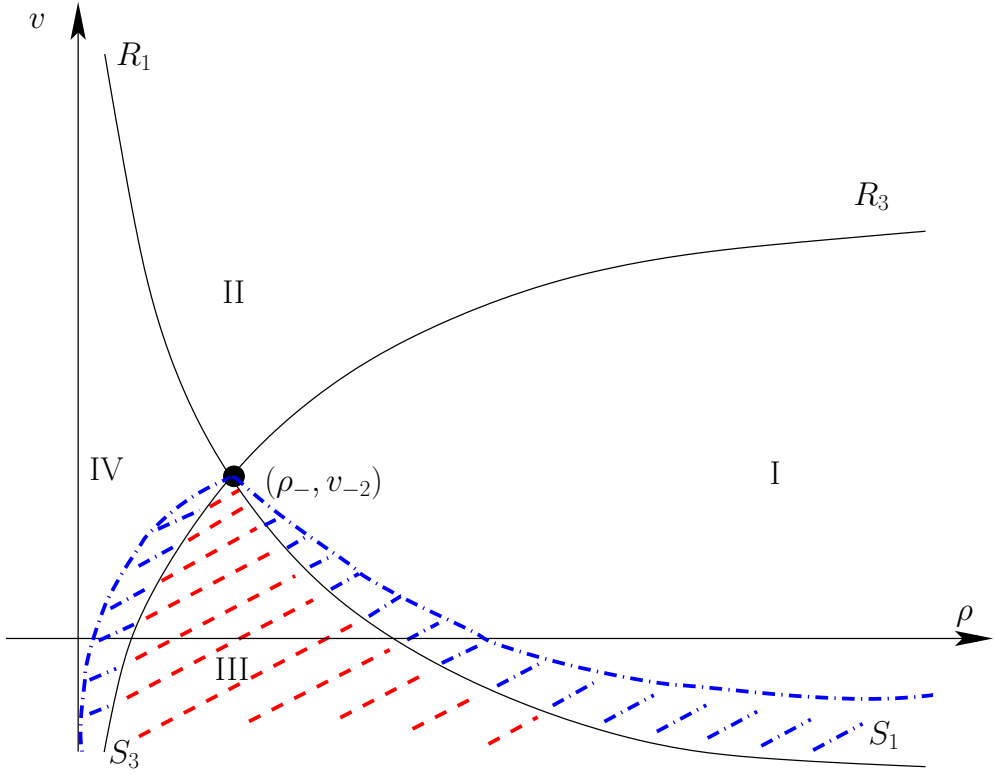}
\caption{Regions where non--uniqueness holds.}
\label{fig:region}
\end{center}
\end{figure}

\section{Preliminaries}

We start with three important definitions taken from \cite{ChDLKr}.

\begin{definition}[Fan partition]\label{d:fan}
A {\em fan partition} of $\R^2\times (0, \infty)$ consists of three open sets $P_-, P_1, P_+$
of the following form 
\begin{align}
 P_- &= \{(x,t): t>0 \quad \mbox{and} \quad x_2 < \nu_- t\}\\
 P_1 &= \{(x,t): t>0 \quad \mbox{and} \quad \nu_- t < x_2 < \nu_+ t\}\\
 P_+ &= \{(x,t): t>0 \quad \mbox{and} \quad x_2 > \nu_+ t\},
\end{align}
where $\nu_- < \nu_+$ is an arbitrary couple of real numbers.
\end{definition}

\begin{definition}[Fan subsolution] \label{d:subs}
A {\em fan subsolution} to the compressible Euler equations \eqref{eq:Euler system} with
initial data \eqref{eq:R_data} is a triple 
$(\overline{\rho}, \overline{v}, \overline{u}): \R^2\times 
(0,\infty) \rightarrow (\R^+, \R^2, \Sym_0^{2\times2})$ of piecewise constant functions satisfying
the following requirements.
\begin{itemize}
\item[(i)] There is a fan partition $P_-, P_1, P_+$ of $\R^2\times (0, \infty)$ such that
\[
(\overline{\rho}, \overline{v}, \overline{u})=  
(\rho_-, v_-, u_-) \bm{1}_{P_-}
+ (\rho_1, v_1, u_1) \bm{1}_{P_1}
+ (\rho_+, v_+, u_+) \bm{1}_{P_+}
\]
where $\rho_1, v_1, u_1$ are constants with $\rho_1 >0$ and $u_\pm =
v_\pm\otimes v_\pm - \textstyle{\frac{1}{2}} |v_\pm|^2 \id$;
\item[(ii)] There exists a positive constant $C$ such that
\begin{equation} \label{eq:subsolution 2}
v_1\otimes v_1 - u_1 < \frac{C}{2} \id\, ;
\end{equation}
\item[(iii)] The triple $(\overline{\rho}, \overline{v}, \overline{u})$ solves the following system in the
sense of distributions:
\begin{align}
&\partial_t \overline{\rho} + {\rm div}_x (\overline{\rho} \, \overline{v}) \;=\; 0\label{eq:continuity}\\
&\partial_t (\overline{\rho} \, \overline{v})+{\rm div}_x \left(\overline{\rho} \, \overline{u} 
\right) + \nabla_x \left( p(\overline{\rho})+\frac{1}{2}\left( C \rho_1
\bm{1}_{P_1} + \overline{\rho} |\overline{v}|^2 \bm{1}_{P_+\cup P_-}\right)\right)= 0.\label{eq:momentum}
\end{align}
\end{itemize}
\end{definition}

\begin{definition}[Admissible fan subsolution]\label{d:admiss}
 A fan subsolution $(\overline{\rho}, \overline{v}, \overline{u})$ is said to be {\em admissible}
if it satisfies the following inequality in the sense of distributions
\begin{align} 
&\de_t \left(\overline{\rho} \varepsilon(\overline{\rho})\right)+\div_x
\left[\left(\overline{\rho}\varepsilon(\overline{\rho})+p(\overline{\rho})\right) \overline{v}\right]
 + \de_t \left( \overline{\rho} \frac{|\overline{v}|^2}{2} \bm{1}_{P_+\cup P_-} \right)
+ \div_x \left(\overline{\rho} \frac{|\overline{v}|^2}{2} \overline{v} \bm{1}_{P_+\cup P_-}\right)\nonumber\\
&\qquad\qquad+ \left[\de_t\left(\rho_1 \, \frac{C}{2} \, \bm{1}_{P_1}\right) 
+ \div_x\left(\rho_1 \, \overline{v} \, \frac{C}{2}  \, \bm{1}_{P_1}\right)\right]
\;\leq\; 0\, .\label{eq:admissible subsolution}
\end{align}
\end{definition}

The existence of infinitely many admissible weak solutions is related to the existence of a single admissible fan subsolution through the following proposition.

\begin{proposition}\label{p:subs}
Let $p$ be any $C^1$ function and $(\rho_\pm, v_\pm)$ be such that there exists at least one
admissible fan subsolution $(\overline{\rho}, \overline{v}, \overline{u})$ of \eqref{eq:Euler system}
with initial data \eqref{eq:R_data}. Then there are infinitely 
many bounded admissible solutions $(\rho, v)$ to \eqref{eq:Euler system},\eqref{eq:energy inequality}, \eqref{eq:R_data} such that 
$\rho=\overline{\rho}$ and $\abs{v}^2\bm{1}_{P_1} = C$.
\end{proposition}

The core of the proof of Proposition \ref{p:subs} is the following fundamental Lemma.
\begin{lemma}\label{l:ci}
Let $(\tilde{v}, \tilde{u})\in \R^2\times \Sym_0^{2\times 2}$ and $C_0>0$ be such that $\tilde{v}\otimes \tilde{v}
- \tilde{u} < \frac{C_0}{2} \id$. For any open set $\Omega\subset \R^2\times \R$ there are infinitely many maps
$(\underline{v}, \underline{u}) \in L^\infty (\R^2\times \R , \R^2\times \Sym_0^{2\times 2})$ with the following property
\begin{itemize}
\item[(i)] $\underline{v}$ and $\underline{u}$ vanish identically outside $\Omega$;
\item[(ii)] $\div_x \underline{v} = 0$ and $\partial_t \underline{v} + \div_x \underline{u} = 0$;
\item[(iii)] $ (\tilde{v} + \underline{v})\otimes (\tilde{v} + \underline{v}) - (\tilde{u} + \underline{u}) = \frac{C_0}{2} \id$
a.e. on $\Omega$.
\end{itemize}
\end{lemma}

Having Lemma \ref{l:ci} at hand, solutions to the Euler equations \eqref{eq:Euler system} 
are created by adding to the single subsolution infinitely many maps as in Lemma \ref{l:ci} in the region $P_1$. 
More precisely we use the Lemma \ref{l:ci} with $\Omega = P_1, (\tilde{v}, \tilde{u}) = (v_1,u_1)$ and $C_0 = C$. 
One can easily check that each couple $(\overline{\rho}, \overline{v} + \underline{v})$ is indeed an admissible weak solution to \eqref{eq:Euler system}.
For a complete proof of Proposition \ref{p:subs}, we refer to \cite[Section 3.3]{ChDLKr}.

The proof of the Lemma \ref{l:ci} can be found in \cite[Section 4]{ChDLKr} and is essentially based on the theory of De Lellis and Sz\'ekelyhidi \cite{dls2} 
for the incompressible Euler system. We will not present the proof here.

As we explained above, our goal is now to find an admissible fan subsolution for Riemann initial data as in Theorem \ref{t:main}. In order to do that, similarly as in \cite{ChKr} we introduce the real numbers 
$\alpha, \beta, \gamma_1, \gamma_2, v_{-1}, v_{-2}, v_{+1}, v_{+2}$ such that
\begin{align} 
v_1 &= (\alpha, \beta),\label{eq:v1}\\
v_- &= (v_{-1}, v_{-2})\\
v_+ &= (v_{+1}, v_{+2})\\
u_1 &=\left( \begin{array}{cc}
    \gamma_1 & \gamma_2 \\
    \gamma_2 & -\gamma_1\\
    \end{array} \right)\, .\label{eq:u1}
\end{align}

As in \cite{ChKr}, we can easily check the existence of a fan subsolution thanks to the following Proposition.

\begin{proposition}\label{p:algebra}
Let $P_-, P_1, P_+$ be a fan partition as in Definition \ref{d:fan}. The constants 
$v_1, v_-, v_+, u_1, \rho_-, \rho_+, \rho_1$ as in \eqref{eq:v1}-\eqref{eq:u1} define an admissible
fan subsolution as in Definitions \ref{d:subs}-\ref{d:admiss} if and only if the following
identities and inequalities hold:
\begin{itemize}
\item Rankine-Hugoniot conditions on the left interface:
\begin{align}
&\nu_- (\rho_- - \rho_1) \, =\,  \rho_- v_{-2} -\rho_1  \beta \label{eq:cont_left}  \\
&\nu_- (\rho_- v_{-1}- \rho_1 \alpha) \, = \, \rho_- v_{-1} v_{-2}- \rho_1 \gamma_2  \label{eq:mom_1_left}\\
&\nu_- (\rho_- v_{-2}- \rho_1 \beta) \, = \,  
\rho_- v_{-2}^2 + \rho_1 \gamma_1 +p (\rho_-)-p (\rho_1) - \rho_1 \frac{C}{2}\, ;\label{eq:mom_2_left}
\end{align}
\item Rankine-Hugoniot conditions on the right interface:
\begin{align}
&\nu_+ (\rho_1-\rho_+ ) \, =\,  \rho_1  \beta - \rho_+ v_{+2} \label{eq:cont_right}\\
&\nu_+ (\rho_1 \alpha- \rho_+ v_{+1}) \, = \, \rho_1 \gamma_2 - \rho_+ v_{+1} v_{+2} \label{eq:mom_1_right}\\
&\nu_+ (\rho_1 \beta- \rho_+ v_{+2}) \, = \, - \rho_1 \gamma_1 - \rho_+ v_{+2}^2 +p (\rho_1) -p (\rho_+) 
+ \rho_1 \frac{C}{2}\, ;\label{eq:mom_2_right}
\end{align}
\item Subsolution condition:
\begin{align}
 &\alpha^2 +\beta^2 < C \label{eq:sub_trace}\\
& \left( \frac{C}{2} -{\alpha}^2 +\gamma_1 \right) \left( \frac{C}{2} -{\beta}^2 -\gamma_1 \right) - 
\left( \gamma_2 - \alpha \beta \right)^2 >0\, ;\label{eq:sub_det}
\end{align}
\item Admissibility condition on the left interface:
\begin{align}
& \nu_-(\rho_- \varepsilon(\rho_-)- \rho_1 \varepsilon( \rho_1))+\nu_- 
\left(\rho_- \frac{\abs{v_-}^2}{2}- \rho_1 \frac{C}{2}\right)\nonumber\\
\leq & \left[(\rho_- \varepsilon(\rho_-)+ p(\rho_-)) v_{-2}- 
( \rho_1 \varepsilon( \rho_1)+ p(\rho_1)) \beta \right] 
+ \left( \rho_- v_{-2} \frac{\abs{v_-}^2}{2}- \rho_1 \beta \frac{C}{2}\right)\, ;\label{eq:E_left}
\end{align}
\item Admissibility condition on the right interface:
\begin{align}
&\nu_+(\rho_1 \varepsilon( \rho_1)- \rho_+ \varepsilon(\rho_+))+\nu_+ 
\left( \rho_1 \frac{C}{2}- \rho_+ \frac{\abs{v_+}^2}{2}\right)\nonumber\\
\leq &\left[ ( \rho_1 \varepsilon( \rho_1)+ p(\rho_1)) \beta- (\rho_+ \varepsilon(\rho_+)+ p(\rho_+)) v_{+2}\right] 
+ \left( \rho_1 \beta \frac{C}{2}- \rho_+ v_{+2} \frac{\abs{v_+}^2}{2}\right)\, .\label{eq:E_right}
\end{align}
\end{itemize}
\end{proposition}

The existence of a fan subsolution is then equivalent to the existence of real numbers $\nu_- < \nu_+, \rho_1 > 0, \alpha, \beta, \gamma_1, \gamma_2, C > 0$ solving the set of identities and inequalities \eqref{eq:cont_left}--\eqref{eq:E_right}. We start with the following observation.

\begin{lemma}\label{l:algebra simplified}
Let $v_{-1} = v_{+1}$. Then $\alpha = v_{-1} = v_{+1}$ and $\gamma_2 = \alpha\beta$.
\end{lemma}
\begin{proof}
See \cite[Lemma 4.2]{ChKr}
\end{proof}

The set of identities and inequalities from Proposition \ref{p:algebra} then simplifies as follows
\begin{itemize}
\item Rankine-Hugoniot conditions on the left interface:
\begin{align}
&\nu_- (\rho_- - \rho_1) \, =\,  \rho_- v_{-2} -\rho_1  \beta \label{eq:cont_left s}  \\
&\nu_- (\rho_- v_{-2}- \rho_1 \beta) \, = \,  
\rho_- v_{-2}^2 - \rho_1(\frac{C}{2} - \gamma_1) +p (\rho_-)-p (\rho_1) \, ;\label{eq:mom_2_left s}
\end{align}
\item Rankine-Hugoniot conditions on the right interface:
\begin{align}
&\nu_+ (\rho_1-\rho_+ ) \, =\,  \rho_1  \beta - \rho_+ v_{+2} \label{eq:cont_right s}\\
&\nu_+ (\rho_1 \beta- \rho_+ v_{+2}) \, = \, \rho_1 (\frac{C}{2} - \gamma_1) - \rho_+ v_{+2}^2 +p (\rho_1) -p (\rho_+) 
\, ;\label{eq:mom_2_right s}
\end{align}
\item Subsolution condition:
\begin{align}
 &\alpha^2 +\beta^2 < C \label{eq:sub_trace s}\\
& \left( \frac{C}{2} -{\alpha}^2 +\gamma_1 \right) \left( \frac{C}{2} -{\beta}^2 -\gamma_1 \right) >0\, ;\label{eq:sub_det s}
\end{align}
\end{itemize}
with admissibility conditions \eqref{eq:E_left} and \eqref{eq:E_right} same as above and $\alpha = v_{-1} = v_{+1}$. Next we reformulate the conditions for the matrix $u_1 + \frac C2 \id - v_1\otimes v_1$ to be positive definite.

\begin{lemma}\label{l:algebra simplified 2}
 A necessary condition for \eqref{eq:sub_trace s}-\eqref{eq:sub_det s} to be satisfied is $\frac{C}{2} -\gamma_1 > \beta^2$. 
\end{lemma}
\begin{proof}
See \cite[Lemma 4.3]{ChKr}
\end{proof}

Following the strategy developed in \cite{ChKr}, we introduce $\ep_1$ and $\ep_2$ as
\begin{align}
&0 < \ep_1 := \frac{C}{2} - \gamma_1 - \beta^2 \\
&0 < \ep_2 := C-\alpha^2-\beta^2-\ep_1
\end{align}
and further reformulate the set of identities and inequalities as follows.

\begin{lemma} \label{l:final lemma}
 In the case $v_{-1} = v_{+1} = \alpha$ and with $\ep_1$, $\ep_2$ as defined above, the set of algebraic identities and inequalities \eqref{eq:cont_left s}-\eqref{eq:sub_det s} together with \eqref{eq:E_left}-\eqref{eq:E_right} is equivalent to 
\begin{itemize}
\item Rankine-Hugoniot conditions on the left interface:
\begin{align}
&\nu_- (\rho_- - \rho_1) \, =\,  \rho_- v_{-2} -\rho_1  \beta \label{eq:cont_left ss}  \\
&\nu_- (\rho_- v_{-2}- \rho_1 \beta) \, = \,  
\rho_- v_{-2}^2 - \rho_1(\beta^2 + \ep_1) +p (\rho_-)-p (\rho_1) \, ;\label{eq:mom_2_left ss}
\end{align}
\item Rankine-Hugoniot conditions on the right interface:
\begin{align}
&\nu_+ (\rho_1-\rho_+ ) \, =\,  \rho_1  \beta - \rho_+ v_{+2} \label{eq:cont_right ss}\\
&\nu_+ (\rho_1 \beta- \rho_+ v_{+2}) \, = \, \rho_1 (\beta^2 + \ep_1) - \rho_+ v_{+2}^2 +p (\rho_1) -p (\rho_+) 
\, ;\label{eq:mom_2_right ss}
\end{align}
\item Subsolution condition:
\begin{align}
& \ep_1 > 0 \label{eq:sub_1 ss}\\
& \ep_2 > 0\, ;\label{eq:sub_2 ss}
\end{align}
\item Admissibility condition on the left interface:
\begin{align}
&(\beta-v_{-2})\left(p(\rho_-)+p(\rho_1)-2\rho_-\rho_1\frac{\ep(\rho_-)-\ep(\rho_1)}{\rho_--\rho_1}\right) \nonumber\\
\leq &\ep_1\rho_1(v_{-2}+\beta) - (\ep_1+\ep_2)\frac{\rho_-\rho_1(\beta-v_{-2})}{\rho_--\rho_1}\, ;\label{eq:E_left ss}
\end{align}
\item Admissibility condition on the right interface:
\begin{align}
&(v_{+2}-\beta)\left(p(\rho_1)+p(\rho_+)-2\rho_1\rho_+\frac{\ep(\rho_1)-\ep(\rho_+)}{\rho_1-\rho_+}\right) \nonumber\\
\leq &-\ep_1\rho_1(v_{+2}+\beta) + (\ep_1+\ep_2)\frac{\rho_1\rho_+(v_{+2}-\beta)}{\rho_1-\rho_+}\, .\label{eq:E_right ss}
\end{align}
\end{itemize}
\end{lemma}

\begin{proof}
 See \cite[Lemma 4.4]{ChKr}
\end{proof}

Let us emphasize that the expressions 
\begin{align}
 P(\rho_-,\rho_1) := \left(p(\rho_-)+p(\rho_1)-2\rho_-\rho_1\frac{\ep(\rho_-)-\ep(\rho_1)}{\rho_--\rho_1}\right) \\
 P(\rho_1,\rho_+) := \left(p(\rho_1)+p(\rho_+)-2\rho_1\rho_+\frac{\ep(\rho_1)-\ep(\rho_+)}{\rho_1-\rho_+}\right)
\end{align}
appearing on the left hand sides of \eqref{eq:E_left ss} and \eqref{eq:E_right ss} are both positive for $p(\rho) = \rho^\gamma$ with $\gamma \geq 1$ as a consequence of \cite[Lemma 2.1]{ChKr}.

\section{Proof}

After reducing the existence of a fan subsolution to Lemma \ref{l:final lemma} we are now ready to prove Theorem \ref{t:main}.

Recall that the quantities $\rho_\pm, v_{\pm 2}$ are considered to be given as the initial data. 
Therefore the system of relations \eqref{eq:cont_left ss}--\eqref{eq:E_right ss} consists of 4 equations and 4 
inequalities for 6 unknowns $\nu_\pm, \rho_1, \beta, \ep_1, \ep_2$, with $\ep_2$ appearing only in the inequalities. 
Similarly as in \cite{ChKr} we choose $\rho_1$ as a parameter and using the equations \eqref{eq:cont_left ss}--\eqref{eq:mom_2_right ss} 
we express $\nu_\pm, \beta$ and $\ep_1$ in terms of the initial data and of the chosen parameter $\rho_1$.

We use the following notation for functions of initial data
\begin{align}
 R &:= \rho_--\rho_+ \label{eq:R}\\
 A &:= \rho_-v_{-2}-\rho_+v_{+2} \label{eq:A}\\
 H &:= \rho_-v_{-2}^2-\rho_+v_{+2}^2 + p(\rho_-) - p(\rho_+). \label{eq:H} \\
 u &:= v_{+2} - v_{-2} \label{eq:u} \\
 B &:= A^2 - RH = \rho_-\rho_+ u^2 - (\rho_+ - \rho_-)(p(\rho_+) - p(\rho_-)) \label{eq:B}
\end{align}
and recall that we study the problem with 
\begin{equation}\label{eq:condition}
v_{-2} - v_{+2} < \sqrt{\frac{(\rho_+ - \rho_-)(p(\rho_+) - p(\rho_-))}{\rho_+\rho_-}}
\end{equation}
which translates into $B < 0$. 

\subsection{The case $\mathbf{R < 0}$}

Let us first assume that $\rho_+ > \rho_-$, i.e. $R < 0$.

Following the calculations of \cite[Section 4]{ChKr}, we recover
\begin{equation}\label{eq:num 1}
 \nu_- = \frac{A - \nu_+(\rho_1-\rho_+)}{\rho_--\rho_1}.
\end{equation}
and
\begin{equation}\label{eq:blabla12}
\nu_+ = \frac{A}{R} \pm \frac{1}{R}\sqrt{B\frac{\rho_1-\rho_-}{\rho_1-\rho_+}}.
\end{equation}
with the correct sign chosen such that $\nu_- < \nu_+$. Since we assume $B < 0$, the necessary condition to follow is $\rho_1 \in (\rho_-,\rho_+)$ and we find that
\begin{align}
 \nu_- = \frac{A}{R} + \frac{\sqrt{-B}}{R}\sqrt{\frac{\rho_+-\rho_1}{\rho_1-\rho_-}} \label{eq:num}\\
 \nu_+ = \frac{A}{R} - \frac{\sqrt{-B}}{R}\sqrt{\frac{\rho_1-\rho_-}{\rho_+-\rho_1}} \label{eq:nup}.
\end{align}

Further we express $\beta$ from \eqref{eq:cont_right ss} as 
\begin{equation}\label{eq:beta}
 \beta = \frac{\rho_+v_{+2}}{\rho_1} - \frac{(\rho_+-\rho_1)A}{R\rho_1} + \frac{\sqrt{-B}}{R\rho_1}\sqrt{(\rho_1-\rho_-)(\rho_+-\rho_1)}
\end{equation}
and finally we use \eqref{eq:mom_2_right ss} to express $\ep_1$ as a function of $\rho_1$ and of the initial data. We have
\begin{equation}\label{eq:ep1}
 \ep_1(\rho_1) = \frac{p(\rho_+) - p(\rho_1)}{\rho_1} - \frac{\rho_+(\rho_+-\rho_1)}{\rho_1^2}(\nu_+ - v_{+2})^2
\end{equation}
and further plugging in \eqref{eq:nup} we get
\begin{equation}\label{eq:ep10}
 \ep_1(\rho_1) = \frac{p(\rho_+) - p(\rho_1)}{\rho_1} - \frac{\rho_+}{\rho_1}\left(\frac{\sqrt{-B}}{-R}\sqrt{1-\frac{\rho_-}{\rho_1}} - \frac{\rho_- u}{R}\sqrt{\frac{\rho_+}{\rho_1}-1}\right)^2
\end{equation}
For simplicity we further denote
\begin{align}
 K &:= \frac{\rho_- u}{R} \label{eq:K}\\
 L &:= \frac{\sqrt{-B}}{-R} \label{eq:L}
\end{align}
and recall that both $K,L > 0$. Thus we have
\begin{equation}\label{eq:ep11}
 \ep_1(\rho_1) = \frac{p(\rho_+) - p(\rho_1)}{\rho_1} - \frac{\rho_+}{\rho_1}\left(L\sqrt{1-\frac{\rho_-}{\rho_1}} - K\sqrt{\frac{\rho_+}{\rho_1}-1}\right)^2
\end{equation}

\begin{lemma}\label{l:ep1 decreasing}
 There exists a unique $\overline{\rho}$ such that 
 \begin{align}
  &\ep_1 > 0 \qquad \text{ for } \quad \rho_1 \in (\rho_-,\overline{\rho}) \\
  &\ep_1 < 0 \qquad \text{ for } \quad \rho_1 \in (\overline{\rho},\rho_+).
 \end{align}
 Moreover, $\overline{\rho} \rightarrow \rho_+$ as $u \rightarrow - \sqrt{\frac{(\rho_+ - \rho_-)(p(\rho_+) - p(\rho_-))}{\rho_+\rho_-}}$
\end{lemma}

\begin{proof}
First observe that $\ep_1(\rho_-) >0$. Indeed we have
\begin{equation*}
\ep_1(\rho_-) = \frac{p(\rho_+) - p(\rho_-)}{\rho_-} - K^2\frac{\rho_+(\rho_+-\rho_-)}{\rho_-^2} = \frac{p(\rho_+) - p(\rho_-)}{\rho_-} - \frac{\rho_+u^2}{\rho_+-\rho_-}
\end{equation*}
and this expression is positive due to \eqref{eq:condition}. Next it is easy to see that $\ep_1(\rho_+) < 0$.

We denote 
\begin{equation}
\tilde{\rho} := \frac{K^2\rho_+ + L^2\rho_-}{K^2+L^2} \in (\rho_-,\rho_+),
\end{equation}
i.e. $\tilde{\rho}$ is the zero of the expression $L\sqrt{1-\frac{\rho_-}{\rho_1}} - K\sqrt{\frac{\rho_+}{\rho_1}-1}$. Obviously $\ep_1(\tilde{\rho}) > 0$. It is easy to observe that the function $\ep_1(\rho_1)$ is decreasing on the interval $(\tilde{\rho},\rho_+)$ thus yielding the existence of a single zero of $\ep_1$ on the interval $(\tilde{\rho},\rho_+)$ which is indeed the $\overline{\rho}$ claimed in the Lemma.

Our final goal is thus to ensure that there are no zeros of $\ep_1(\rho_1)$ on the interval $(\rho_-,\tilde{\rho})$. This is equivalent to say that there are no zeros of the function $\tilde{\ep_1}(\rho_1) = \rho_1\ep_1(\rho_1)$ on this interval. For this purpose it is enough to show that $\tilde{\ep_1}(\rho_1)$ is a concave function on $(\rho_-,\tilde{\rho})$.

We have
\[
\tilde{\ep_1}(\rho_1) = p(\rho_+) - p(\rho_1) - \rho_+\left(L\sqrt{1-\frac{\rho_-}{\rho_1}} - K\sqrt{\frac{\rho_+}{\rho_1}-1}\right)^2
\]
\[
\tilde{\ep_1}'(\rho_1) = -p'(\rho_1) + \rho_+\rho_1^{-2}\left(K^2\rho_+-L^2\rho_- + KL\left(\rho_-\sqrt{\frac{\rho_+-\rho_1}{\rho_1-\rho_-}} - \rho_+\sqrt{\frac{\rho_1-\rho_-}{\rho_+-\rho_1}}\right)\right)
\]
and finally
\begin{align}
\nonumber \tilde{\ep_1}''(\rho_1) &= -p''(\rho_1) -2\rho_+\rho_1^{-3}\left(K^2\rho_+-L^2\rho_- + KL\left(\rho_-\sqrt{\frac{\rho_+-\rho_1}{\rho_1-\rho_-}} - \rho_+\sqrt{\frac{\rho_1-\rho_-}{\rho_+-\rho_1}}\right)\right) \\
 &+ \frac{\rho_+KL}{2\rho_1^2}\left(-\frac{\rho_-(\rho_+-\rho_-)}{(\rho_1-\rho_-)^2}\sqrt{\frac{\rho_1-\rho_-}{\rho_+-\rho_1}} - \frac{\rho_+(\rho_+-\rho_-)}{(\rho_+-\rho_1)^2}\sqrt{\frac{\rho_+-\rho_1}{\rho_1-\rho_-}}\right). 
\end{align}
The first and third terms on the right hand side are clearly nonpositive, so to conclude our proof we need to show that 
\[
K^2\rho_+-L^2\rho_- + KL\left(\rho_-\sqrt{\frac{\rho_+-\rho_1}{\rho_1-\rho_-}} - \rho_+\sqrt{\frac{\rho_1-\rho_-}{\rho_+-\rho_1}}\right) > 0
\]
for $\rho_1 \in (\rho_-,\tilde{\rho})$. Denoting
\[
x = \sqrt{\frac{\rho_+-\rho_1}{\rho_1-\rho_-}} \geq 0
\]
we search where  
\[
K^2\rho_+-L^2\rho_- + KL\rho_-x - KL\rho_+x^{-1} > 0.
\]
Simple computations yield that this is satisfied for $x > L/K$ which means $\rho_1 < \tilde{\rho}$. The proof of the first claim of Lemma \ref{l:ep1 decreasing} is complete. 

In order to prove the second claim we observe that $L \rightarrow 0$ and $\ep_1(\rho_+) \rightarrow 0$ as $u \rightarrow - \sqrt{\frac{(\rho_+ - \rho_-)(p(\rho_+) - p(\rho_-))}{\rho_+\rho_-}}$. In particular also $\tilde{\rho} \rightarrow \rho_+$ and since obviously $\tilde{\rho} < \overline{\rho} < \rho_+$ we conclude that $\overline{\rho} \rightarrow \rho_+$.
\end{proof}
\\

It remains to study the admissibility conditions. Let us denote 
\begin{equation}\label{eq:P def}
 P(r,s) := p(r)+p(s)-2rs\frac{\ep(r)-\ep(s)}{r-s}
\end{equation}
and thus the admissibility conditions \eqref{eq:E_left ss}--\eqref{eq:E_right ss} can be rewritten as follows
\begin{align}
(\beta-v_{-2})P(\rho_-,\rho_1) &\leq \ep_1\rho_1(v_{-2}+\beta) - (\ep_1+\ep_2)\frac{\rho_-\rho_1(\beta-v_{-2})}{\rho_--\rho_1} \label{eq:AD1}\\ 
(v_{+2}-\beta)P(\rho_1,\rho_+) &\leq -\ep_1\rho_1(v_{+2}+\beta) + (\ep_1+\ep_2)\frac{\rho_1\rho_+(v_{+2}-\beta)}{\rho_1-\rho_+}.\label{eq:AD2}
\end{align}

To prove the main theorem, we will now rewrite some of the above expressions in a different way. We denote
\begin{equation}\label{eq:T}
T = \frac{(\rho_+-\rho_-)(p(\rho_+)-p(\rho_-))}{\rho_+\rho_-}
\end{equation}
and recall that we are interested in the cases where $u+\sqrt{T} > 0$. In particular we will study the limits as $u \rightarrow -\sqrt{T}$.

We have
\[
B = \rho_-\rho_+(u^2-T) = \rho_-\rho_+(u-\sqrt{T})(u+\sqrt{T}),
\]
where the middle term is negative. Rewriting \eqref{eq:num} and \eqref{eq:nup} we get
\begin{align}
\nu_- = v_{+2} + \frac{\rho_-\sqrt{T}}{R} - \frac{\rho_-}{R}(u+\sqrt{T}) + \sqrt{u+\sqrt{T}}\frac{\sqrt{(\sqrt{T}-u)\rho_-\rho_+}}{R}\sqrt{\frac{\rho_+-\rho_1}{\rho_1-\rho_-}} \label{eq:num2}\\
\nu_+ = v_{+2} + \frac{\rho_-\sqrt{T}}{R} - \frac{\rho_-}{R}(u+\sqrt{T}) - \sqrt{u+\sqrt{T}}\frac{\sqrt{(\sqrt{T}-u)\rho_-\rho_+}}{R}\sqrt{\frac{\rho_1-\rho_-}{\rho_+-\rho_1}} \label{eq:nup2}.
\end{align}
In particular we observe that for any positive $u+\sqrt{T}$ we have $\nu_- < \nu_+$, but both $\nu_-,\nu_+$ have the same limit as $u \rightarrow -\sqrt{T}$.

Next we reformulate the expression \eqref{eq:beta}. We have
\begin{align}\label{eq:beta2}
 \beta &= v_{+2} + \frac{\rho_-(\rho_1-\rho_+)\sqrt{T}}{R\rho_1} - (u+\sqrt{T})\frac{\rho_-(\rho_1-\rho_+)}{R\rho_1} \\ \nonumber
 &+ \sqrt{u+\sqrt{T}}\frac{\sqrt{(\sqrt{T}-u)\rho_-\rho_+}}{R\rho_1}\sqrt{(\rho_1-\rho_-)(\rho_+-\rho_1)}
\end{align}
and denoting $\overline{\beta}$ = $\lim_{u\rightarrow-\sqrt{T}} \beta$ we observe
\begin{equation}\label{eq:betabar}
 \overline{\beta} = v_{+2} + \frac{\rho_-(\rho_1-\rho_+)\sqrt{T}}{R\rho_1}
\end{equation}

We will not rewrite in detail the expression \eqref{eq:ep11} for $\ep_1$ and instead we introduce $\overline{\ep_1} = \lim_{u\rightarrow-\sqrt{T}} \ep_1$
\begin{equation}\label{eq:epbar}
\overline{\ep_1} = \frac{p(\rho_+) - p(\rho_1)}{\rho_1} - \frac{\rho_+\rho_-^2(\rho_+-\rho_1)T}{\rho_1^2 R^2}
\end{equation}
Using Lemma \ref{l:ep1 decreasing} we have that $\overline{\ep_1}(\rho_-) = \overline{\ep_1}(\rho_+) = 0$ and $\overline{\ep_1} > 0$ for all $\rho_1 \in (\rho_-,\rho_+)$.

In order to examine the admissibility inequalities we first study the signs of the expressions $\beta - v_{-2}$ and $v_{+2}-\beta$. Recalling $R < 0$ we see from \eqref{eq:betabar} that 
\begin{equation}\label{eq:beta_vp}
v_{+2}-\overline{\beta} = - \frac{\rho_-(\rho_1-\rho_+)\sqrt{T}}{R\rho_1},
\end{equation}
so it is negative for $\rho_1 \in (\rho_-,\rho_+)$ with zero value in $\rho_1 = \rho_+$ and a strictly negative value in $\rho_1 = \rho_-$. 
Using a continuity argument we conclude $v_{+2}-\beta < 0$ at least on some interval $(\rho_-,\rho_+-\varepsilon)$ with $\varepsilon$ small for $u$ close to $-\sqrt{T}$.

Similarly we have 
\begin{equation}\label{eq:beta_vm}
\overline{\beta} - v_{-2} = \frac{\rho_+(\rho_1-\rho_-)\sqrt{T}}{R\rho_1}
\end{equation}
and again, this expression is negative for $\rho_1 \in (\rho_-,\rho_+)$ with zero value in $\rho_1 = \rho_+$ and a strictly negative value in $\rho_1 = \rho_-$. 
Thus we conclude $\beta - v_{-2} < 0$ at least on some interval $(\rho_-+\varepsilon,\rho_+)$ with $\varepsilon$ small for $u$ close to $-\sqrt{T}$.

Using this information we reformulate \eqref{eq:AD1} and \eqref{eq:AD2} as 
\begin{align}
\ep_2 &\leq \frac{\rho_1-\rho_-}{\rho_1\rho_-}P(\rho_-,\rho_1) - \ep_1\rho_1 \frac{\beta + v_{-2}}{\beta - v_{-2}}\frac{\rho_1-\rho_-}{\rho_1\rho_-} - \ep_1 := M_1 \label{eq:AD11}\\ 
\ep_2 &\geq -\frac{\rho_+-\rho_1}{\rho_+\rho_1}P(\rho_1,\rho_+) -\ep_1\rho_1\frac{v_{+2}+\beta}{v_{+2}-\beta}\frac{\rho_+-\rho_1}{\rho_+\rho_1} - \ep_1 := M_2.\label{eq:AD21}
\end{align}

To finish our proof we proceed as follows. We show first that in the limit $u \rightarrow -\sqrt{T}$ we have $M_1 > M_2$ for all $\rho_1 \in (\rho_-,\rho_+)$ and then that we can find some $s \in (\rho_-,\rho_+)$ for which $M_1(s) > 0$. By a continuity argument we then conclude that at least for $u$ sufficiently close to $-\sqrt{T}$ we will still have $v_{+2} - \beta(s) < 0$, $\beta(s) - v_{-2} < 0$, $M_1(s) > M_2(s)$ and $M_1(s) > 0$, thus there exist $\ep_2$ satisfying both \eqref{eq:AD11} and \eqref{eq:AD21} while $\ep_1(s) > 0$.

Let us now therefore first express $\overline{M_1} = \lim_{u\rightarrow-\sqrt{T}} M_1$. In order to do this we start with $\frac{\overline{\beta} + v_{-2}}{\overline{\beta} - v_{-2}}$. We have

\begin{align*}
\frac{\overline{\beta} + v_{-2}}{\overline{\beta} - v_{-2}} &= \frac{v_{+2} + v_{-2} + \frac{\rho_-\sqrt{T}(\rho_1-\rho_+)}{\rho_1R}}{v_{+2} - v_{-2} + \frac{\rho_-\sqrt{T}(\rho_1-\rho_+)}{\rho_1R}} = \frac{2v_{+2} + \sqrt{T}\left(\frac{2\rho_--\rho_+}{R} - \frac{\rho_-\rho_+}{\rho_1R}\right)}{\sqrt{T}\frac{\rho_+(\rho_1-\rho_-)}{\rho_1R}} \\
&= \frac{2v_{+2}\rho_1R}{\sqrt{T}\rho_+(\rho_1-\rho_-)} + \frac{(2\rho_--\rho_+)\rho_1 - \rho_-\rho_+}{\rho_+(\rho_1-\rho_-)}
\end{align*}
Plugging this into \eqref{eq:AD11} we obtain
\begin{equation}\label{eq:AD12}
\ep_2 \leq \frac{\rho_1-\rho_-}{\rho_1\rho_-}P(\rho_-,\rho_1) - \frac{\overline{\ep_1}\rho_1}{\rho_-\rho_+}\left(2\rho_- - \rho_+ + \frac{2Rv_{+2}}{\sqrt{T}}\right) = \overline{M_1}
\end{equation}
with $\overline{\ep_1}$ given by \eqref{eq:epbar}. Similarly we handle $\overline{M_2} = \lim_{u\rightarrow-\sqrt{T}} M_2$. We have
\begin{equation*}
\frac{v_{+2} + \overline{\beta}}{v_{+2} - \overline{\beta}} = \frac{2v_{+2} + \frac{\rho_-\sqrt{T}(\rho_1-\rho_+)}{\rho_1R}}{- \frac{\rho_-\sqrt{T}(\rho_1-\rho_+)}{\rho_1R}} = -1 - \frac{2v_{+2}\rho_1R}{\sqrt{T}\rho_-(\rho_1-\rho_+)}
\end{equation*}
and inserting this into \eqref{eq:AD21} we get
\begin{equation}\label{eq:AD22}
\ep_2 \geq -\frac{\rho_+-\rho_1}{\rho_+\rho_1}P(\rho_1,\rho_+) -\frac{\overline{\ep_1}\rho_1}{\rho_-\rho_+}\left(\rho_- + \frac{2Rv_{+2}}{\sqrt{T}}\right) = \overline{M_2}
\end{equation}

Now it is easy to show that $\overline{M_1} > \overline{M_2}$ just by comparing the two expressions as we have
\[
\frac{\rho_1-\rho_-}{\rho_1\rho_-}P(\rho_-,\rho_1) + \frac{\rho_+-\rho_1}{\rho_+\rho_1}P(\rho_1,\rho_+) > \frac{\overline{\ep_1}\rho_1}{\rho_-\rho_+}\left(\rho_--\rho_+\right),
\]
which obviously holds for all $\rho_1 \in (\rho_-,\rho_+)$ since the left hand side is positive and the right hand side is negative. This means we can always find $\ep_2$ satisfying both \eqref{eq:AD12} and \eqref{eq:AD22}.

Next we have to assure that such $\ep_2$ can be chosen positive. 
This, however, follows from the fact that $\overline{M_1}(\rho_+) = \frac{\rho_+-\rho_-}{\rho_+\rho_-}P(\rho_-,\rho_+) > 0$. In particular, we can always find $s < \rho_+$ such that $\overline{M_1}(s) > \frac{\rho_+-\rho_-}{2\rho_+\rho_-}P(\rho_-,\rho_+)$ and then use the continuity argument to show that for $u$ sufficiently close to $-\sqrt{T}$ there exists an admissible fan subsolution and therefore also infinitely many admissible weak solutions to the Euler equations \eqref{eq:Euler system}.

\subsection{The case $\mathbf{ R > 0}$}

Now let us treat the case $R > 0$, i.e. $\rho_- > \rho_+$. Note that in the case $\rho_- = \rho_+$ the self-similar solution to the Riemann problem can consist only of two rarefaction waves or of two admissible shocks, in particular it is not possible for the self-similar solution to consist of one shock and one rarefaction wave.

The case $R > 0$ is on one hand quite similar to the case $R < 0$, on the other hand we have to use different equations to obtain the same result, therefore we emphasize here how to proceed in this case.

First, the expressions for $\nu_+$ and $\nu_-$ have to be modified in order to obtain $\nu_- < \nu_+$. Instead of \eqref{eq:num}-\eqref{eq:nup} we now have
\begin{align}
 \nu_- = \frac{A}{R} - \frac{\sqrt{-B}}{R}\sqrt{\frac{\rho_1-\rho_+}{\rho_--\rho_1}} \label{eq:num22}\\
 \nu_+ = \frac{A}{R} + \frac{\sqrt{-B}}{R}\sqrt{\frac{\rho_--\rho_1}{\rho_1-\rho_+}} \label{eq:nup22}.
\end{align}
In order to express $\beta$ and $\ep_1$ we now choose to work rather with the equation on the left interface \eqref{eq:cont_left ss}-\eqref{eq:mom_2_left ss} than on the right interface \eqref{eq:cont_right ss}-\eqref{eq:mom_2_right ss}. Thus we obtain the expression for $\beta$ as
\begin{equation}\label{eq:beta22}
 \beta = \frac{\rho_-v_{-2}}{\rho_1} - \frac{(\rho_--\rho_1)A}{R\rho_1} + \frac{\sqrt{-B}}{R\rho_1}\sqrt{(\rho_--\rho_1)(\rho_1-\rho_+)}.
\end{equation}
and next 
\begin{equation}\label{eq:ep122}
 \ep_1(\rho_1) = \frac{p(\rho_-) - p(\rho_1)}{\rho_1} - \frac{\rho_-(\rho_--\rho_1)}{\rho_1^2}(\nu_- - v_{-2})^2,
\end{equation}
and consequently using \eqref{eq:num22}
\begin{equation}\label{eq:ep1022}
 \ep_1(\rho_1) = \frac{p(\rho_-) - p(\rho_1)}{\rho_1} - \frac{\rho_-}{\rho_1}\left(\frac{\sqrt{-B}}{R}\sqrt{1-\frac{\rho_+}{\rho_1}} + \frac{\rho_+ u}{R}\sqrt{\frac{\rho_-}{\rho_1}-1}\right)^2.
\end{equation}
Now, similarly as in the case $\rho_- < \rho_+$, we can denote 
\begin{align}
 K &:= \frac{\rho_+ u}{-R} \label{eq:K2}\\
 L &:= \frac{\sqrt{-B}}{R} \label{eq:L2}
\end{align}
so that both $K,L > 0$ and obtain
\begin{equation}\label{eq:ep1122}
 \ep_1(\rho_1) = \frac{p(\rho_-) - p(\rho_1)}{\rho_1} - \frac{\rho_-}{\rho_1}\left(L\sqrt{1-\frac{\rho_+}{\rho_1}} - K\sqrt{\frac{\rho_-}{\rho_1}-1}\right)^2.
\end{equation}
Notice that expression \eqref{eq:ep1122} is the same as \eqref{eq:ep11} just with switched indices $+$ and $-$. In particular, we can use the proof of Lemma \ref{l:ep1 decreasing} to deduce that in the case $\rho_- > \rho_+$ we have
\begin{lemma}\label{l:ep1 decreasing2}
 There exists a unique $\overline{\rho}$ such that 
 \begin{align}
  &\ep_1 > 0 \qquad \text{ for } \quad \rho_1 \in (\rho_+,\overline{\rho}) \\
  &\ep_1 < 0 \qquad \text{ for } \quad \rho_1 \in (\overline{\rho},\rho_-).
 \end{align}
 Moreover, $\overline{\rho} \rightarrow \rho_-$ as $u \rightarrow - \sqrt{\frac{(\rho_- - \rho_+)(p(\rho_-) - p(\rho_+))}{\rho_+\rho_-}}$.
\end{lemma}

Concerning the admissibility conditions, the expressions \eqref{eq:AD1}-\eqref{eq:AD2} stay exactly the same. Instead of \eqref{eq:num2}-\eqref{eq:nup2} we have
\begin{align}
\nu_- = v_{+2} + \frac{\rho_-\sqrt{T}}{R} - \frac{\rho_-}{R}(u+\sqrt{T}) - \sqrt{u+\sqrt{T}}\frac{\sqrt{(\sqrt{T}-u)\rho_-\rho_+}}{R}\sqrt{\frac{\rho_1-\rho_+}{\rho_--\rho_1}} \label{eq:num3}\\
\nu_+ = v_{+2} + \frac{\rho_-\sqrt{T}}{R} - \frac{\rho_-}{R}(u+\sqrt{T}) + \sqrt{u+\sqrt{T}}\frac{\sqrt{(\sqrt{T}-u)\rho_-\rho_+}}{R}\sqrt{\frac{\rho_--\rho_1}{\rho_1-\rho_+}} \label{eq:nup3}.
\end{align}
The expression \eqref{eq:beta2} for $\beta$ actually stays the same
\begin{align}\label{eq:beta3}
 \beta &= v_{+2} + \frac{\rho_-(\rho_1-\rho_+)\sqrt{T}}{R\rho_1} - (u+\sqrt{T})\frac{\rho_-(\rho_1-\rho_+)}{R\rho_1} \\ \nonumber
 &+ \sqrt{u+\sqrt{T}}\frac{\sqrt{(\sqrt{T}-u)\rho_-\rho_+}}{R\rho_1}\sqrt{(\rho_1-\rho_-)(\rho_+-\rho_1)}
\end{align}
and in particular we also have $\overline{\beta} = \lim_{u\rightarrow-\sqrt{T}} \beta$ as
\begin{equation}\label{eq:betabar2}
 \overline{\beta} = v_{+2} + \frac{\rho_-(\rho_1-\rho_+)\sqrt{T}}{R\rho_1}.
\end{equation}
Next we express $\overline{\ep_1} = \lim_{u\rightarrow-\sqrt{T}} \ep_1$ as
\begin{equation}\label{eq:epbar2}
\overline{\ep_1} = \frac{p(\rho_-) - p(\rho_1)}{\rho_1} - \frac{\rho_-\rho_+^2(\rho_--\rho_1)T}{\rho_1^2 R^2}
\end{equation}
and again use Lemma \ref{l:ep1 decreasing2} to observe that $\overline{\ep_1}(\rho_-) = \overline{\ep_1}(\rho_+) = 0$ and $\overline{\ep_1} > 0$ for all $\rho_1 \in (\rho_+,\rho_-)$.

The signs of $v_{+2} - \beta$ and $\beta - v_{-2}$ can be again shown to be negative at least on intervals $(\rho_+ +\varepsilon,\rho_-)$ and $(\rho_+,\rho_- -\varepsilon)$ respectively. Reverse order of $\rho_-$ and $\rho_+$ causes the admissibility conditions to change to

\begin{align}
\ep_2 &\geq -\frac{\rho_--\rho_1}{\rho_1\rho_-}P(\rho_-,\rho_1) + \ep_1\rho_1 \frac{\beta + v_{-2}}{\beta - v_{-2}}\frac{\rho_--\rho_1}{\rho_1\rho_-} - \ep_1 := N_1 \label{eq:AD113}\\ 
\ep_2 &\leq \frac{\rho_1-\rho_+}{\rho_+\rho_1}P(\rho_1,\rho_+) + \ep_1\rho_1\frac{v_{+2}+\beta}{v_{+2}-\beta}\frac{\rho_1-\rho_+}{\rho_+\rho_1} - \ep_1 := N_2.\label{eq:AD213}
\end{align}

The rest of the proof is now following the same steps as in the case $R < 0$ with $N_2$ playing the role of $M_1$ and vice versa. We show that $\overline{N_2} > \overline{N_1}$ and since for $\rho_1 = \rho_-$ we have $\overline{N_2} > 0$ we again conclude the existence of an admissible subsolution. The proof is finished.

\section{Concluding remarks}

We mention here some more remarks about the problem, concentrating on the case $R < 0$. The consequences of the admissibility inequalities \eqref{eq:AD1}-\eqref{eq:AD2} heavily depend on the signs of the expressions $\beta - v_{-2}$ and $v_{+2}-\beta$. Whereas it can be shown that $\beta - v_{-2}$ is always negative (not only in the limit $\overline{\beta}-v_{-2}$ as was shown in \eqref{eq:beta_vm}), this is not the case of $v_{+2}-\beta$, in fact we have $v_{+2}-\beta < 0$ on $(\rho_-,\rho_T)$ and $v_{+2}-\beta > 0$ on $(\rho_T,\rho_+)$, where $\rho_T = \frac{\rho_-\rho_+T}{\rho_- u^2 + \rho_+(T - u^2)}$. In particular for fixed $u > -\sqrt{T}$ one can search for a subsolution in two regions. On the interval $(\rho_-,\rho_T)$ the admissibility condition \eqref{eq:AD2} transfers to \eqref{eq:AD21} as stated in the previous section. However, on the interval $(\rho_T,\rho_+)$ the sign in \eqref{eq:AD21} is opposite and $M_2$ becomes the upper bound for $\ep_2$, not a lower bound.

Note also that the special case $v_{+2} = \beta$, i.e. $\rho = \rho_T$, considerably simplifies the admissibility condition \eqref{eq:AD2}. In particular the inequality \eqref{eq:AD2} becomes just $0 \leq -\ep_1(\rho_T)\rho_T v_{+2}$ which is satisfied if and only if $v_{+2} \leq 0$. In this case the subsolution exists if $M_1(\rho_T) > 0$. As it turns out in the examples below, this is not the optimal strategy and for $v_{+2} < 0$ there are subsolutions with $\rho_1 > \rho_T$ in the case when $M_1(\rho_T) < 0$.

In the case $v_{+2} > 0$, there are no subsolutions with $\rho_1 \in (\rho_T,\rho_+)$ as the expression on the right hand side of \eqref{eq:AD21} is negative on this interval. This no longer holds for $v_{+2} < 0$, where that expression becomes positive at least on some part of the interval $(\rho_T,\rho_+)$ and there may exist subsolutions with such density $\rho_1$.

Finally let us discuss the special case $v_{+2} = 0$. In this case the admissibility inequality \eqref{eq:AD2} simplifies to 
\begin{align}
\ep_2 &\geq -\frac{\rho_+-\rho_1}{\rho_+\rho_1}P(\rho_1,\rho_+) - \frac{\rho_1}{\rho_+}\ep_1 \qquad \text{ for } \rho_1 \in (\rho_-,\rho_T) \label{eq:AD98}\\
\ep_2 &\leq -\frac{\rho_+-\rho_1}{\rho_+\rho_1}P(\rho_1,\rho_+) - \frac{\rho_1}{\rho_+}\ep_1 \qquad \text{ for } \rho_1 \in (\rho_T,\rho_+) \label{eq:AD99}
\end{align}
and is trivially satisfied in $\rho_1 = \rho_T$. In particular it is easy to observe that the expression on the right hand side of \eqref{eq:AD98} and \eqref{eq:AD99} is negative whenever $\ep_1$ is positive. This means that there cannot be any subsolution on $(\rho_T,\rho_+)$, on the other hand the inequality \eqref{eq:AD98} imposes no further restriction on $\ep_2$ on $(\rho_-,\rho_T)$.

In order to illustrate how large is the set of Riemann initial data for which the existence of infinitely many admissible weak solutions is proved in this note, we provide here some examples of Riemann data allowing for existence of infinitely many solutions.

Let us take similarly as the example in \cite{ChDLKr} the pressure law $p(\rho) = \rho^2$ and let $\rho_- = 1$, $\rho_+ = 4$. In this case we have $T = \frac{(\rho_+-\rho_-)(p(\rho_+)-p(\rho_-))}{\rho_+\rho_-} = \frac{45}{4}$ and thus we are interested in Riemann data satisfying $v_{-2} - v_{+2} < \frac{\sqrt{45}}{2} \sim 3.35$. The case of the example in \cite{ChDLKr} was taken as $v_{+2} = 0$ and $v_{-2} = 2\sqrt{2}(\sqrt{\rho_+}-\sqrt{\rho_-}) = 2.83$.

Detail analysis of the problem gives us the following thresholds for various positive values of $v_{+2}$:
\begin{align}
v_{+2} = 0.1 \qquad &\Longrightarrow \qquad V \sim 2.75 \nonumber \\
v_{+2} = 1 \qquad &\Longrightarrow \qquad V \sim 2.955 \nonumber \\
v_{+2} = 2 \qquad &\Longrightarrow \qquad V \sim  3.05 \nonumber 
\end{align}
In this case the subsolution is obtained such that $v_{+2} - \beta$ is indeed negative and there are no subsolutions in the region where $v_{+2} - \beta$ is positive. Moreover it seems that $V$ is increasing with $v_{+2}$.

For $v_{+2} = 0$ the value of $V$ is approximately $2.7$.

For negative values of $v_{+2}$ we get even lower values of $V$ meaning larger set of initial data for which nonuniqueness holds. 
\begin{align}
v_{+2} = -0.1 \qquad &\Longrightarrow \qquad V \sim 2.65 \nonumber \\
v_{+2} = -1 \qquad &\Longrightarrow \qquad V \sim 1.8 \nonumber \\
v_{+2} = -2 \qquad &\Longrightarrow \qquad V \sim  1.02 \nonumber
\end{align}
In this case the subsolutions near the critical values of $V$ are obtained in the region where $v_{+2} - \beta$ is negative. Moreover it seems $V$ is decreasing with $\abs{v_{+2}}$ increasing.

Finally, for $R > 0$ similar remarks hold switching $v_{-2}$ and $v_{+2}$ and several inequality signs.

\section{Acknowledgment}

This work was done mainly during the visit of O.K. at EPFL in January 2017. During the workshop Ideal Fluids and Transport at IMPAN in Warsaw (February 13-15, 2017) the authors learned about similar results achieved by Ch. Klingenberg and S. Markfelder from W\"urzburg University. The authors would like to emphasize that both results were achieved independently on each other.

\end{document}